\documentclass[12pt]{article}
\usepackage{amsmath,amsfonts,amssymb,amsthm,epsfig,epstopdf,titling,url,array}
\usepackage{eqnarray}
\usepackage{enumerate}
\usepackage[a4paper]{geometry}
\usepackage{amscd}

\usepackage{authblk}
\theoremstyle{plain} % default
\newtheorem{theorem}{ Theorem}[section]
\newtheorem{lemma}[theorem]{Lemma}
\newtheorem{proposition}[theorem]{Proposition}
\newtheorem{corollary}[theorem]{Corollary}
\newtheorem{definition}[theorem]{Definition}
\newtheorem{example}[theorem]{Example}
\newtheorem{remark}{\bf Remark}
\newtheorem* {note*}{Note}
\newcommand{\co} {\mathbb{C}}
\newcommand{\R} {\mathbb{R}}

\newcommand{\iy} {\infty}
\newcommand{\N} {\mathbb{N}}
\newcommand{\D} {\mathbb{D}}
\newcommand{\Z} {\mathbb{Z}}

\newcommand{\al}{\alpha}
\newcommand{\lm}{\lambda}
\newcommand{\si}{\sigma}
\newcommand{\eps}{\epsilon}
\newcommand{\h}{\mathcal H}
\newcommand{\A}{\mathcal A}
\newcommand{\BH}{\mathcal B(\mathcal H)}
\newcommand{\suc}{\;\big|\;}
\newcommand{\norm}[1]{\left\Vert#1\right\Vert}
\newcommand{\abs}[1]{\left\vert#1\right\vert}
\newcommand{\set}[1]{\left\{#1\right\}}

\newcommand{\brc}[1]{\left(#1\right)}
\newcommand{\iner}[1]{\left<#1\right>}
\newcommand{\LZ}{\ell^{2}(\mathbb Z)}
\newcommand{\LN}{\ell^{2}(\mathbb N)}
\newcommand{\LNp}{\ell^{p}(\mathbb N)}
\title{\bf \Large A review of some works in the theory of diskcyclic operators}
\bf
\author[1]{\bf\footnotesize Nareen Bamerni \thanks{nareen\_bamerni@yahoo.com}}
\author[1]{\bf Adem K{\i}l{\i}\c{c}man \thanks{akilicman@yahoo.com}}
\author[2]{\bf Mohd Salmi Md Noorani\thanks{msn@ukm.my}}
\affil[1]{\bf Department of Mathematics, University Putra Malaysia,
43400 UPM, Serdang, Selangor, Malaysia}
\affil[2]{\bf School of Mathematical Science, University Kebangsaan Malaysia,
43600 UKM Bangi, Selangor, Malaysia}

\begin{document}
\date{}
\maketitle
\begin{abstract}
In this paper, we give a brief review concerning diskcyclic operators and then we provide some further characterizations of diskcyclic operators on separable Hilbert spaces. In particular, we show that if $x\in \h$ has a disk orbit under $T$ that is somewhere dense in $\h$ then the disk orbit of $x$ under $T$ need not be everywhere dense in $\h$.
We also show that the inverse and  the adjoint of a diskcyclic operator need not be diskcyclic. Moreover, we establish another diskcyclicity criterion and use it to find a necessary and sufficient condition for unilateral backward shifts that are diskcyclic operators. We show that a diskcyclic operator exists on a Hilbert space $\h$ over the field of complex numbers if and only if $\dim(\h)=1$ or $\dim(\h)=\iy$ . Finally we give a sufficient condition for the somewhere density disk orbit to be everywhere dense.\\

\end{abstract}

\section{Introduction.}
In this paper, all Hilbert spaces are separable over the field $\co$ of complex numbers. As usual, $\N$ is the set of all non-negative integers, $\Z$ is the set of all integers and $\BH$ is the space of all continuous linear operators on a Hilbert space $\h$.\\

An operator $T$ is called {\bf hypercyclic} if there is some vector $x\in \h$ such that $Orb(T,x)=\{T^n x: n\in \N\}$ is dense in $\h$, where such a vector $x$ is called {\bf hypercyclic} for $T$. The first example of hypercyclic operator was given by Rolewicz in \cite{Rolewicz}. He proved that if $B$ is a backward shift on the Banach space $\LNp$ then $\lm B$ is hypercyclic for any complex number $\lm$ such that $|\lm|>1$. This leads us to the consideration of scaled orbits. Later, Hilden and  Wallen in \cite{Hilden} undertook the concept of supercyclic operators. An operator $T$ is called {\bf supercyclic} if there is a vector $x\in \h$ such that $\co Orb(T,x)=\{\lm T^n x: \lm \in \co,\, n\in \N\}$ is dense in $\h$, where $x$ is called {\bf supercyclic vector}. For the more detailed information on both hypercyclicity and supercyclicity, the reader may refer to \cite{dynamic,Erdman,Kitai}.\\

In the same spirit, since the operator $\lm B$ is not hypercyclic whenever $|\lm| \le 1$, we are motivated to study the disk orbit. The diskcyclicity phenomenon was introduced by Zeana in her PhD thesis \cite{cyclic}. An operator $T$ is called {\bf diskcyclic} if there is a vector $x\in \h$ such that the set $\D Orb(T,x)=\{\al T^nx:n\geq 0, \al\in \co ,|\al|\leq 1\}$ is dense in $\h$, where the vector $x$ is called {\bf diskcyclic for $T$}. The diskcyclic criterion - a sufficient set of conditions for an operator to be diskcyc- was created by Zeana in \cite{cyclic}, who showed that the diskcyclicity criterion is a mid way between the hypercyclicity and the supercyclicity criterions.\\

The following diagram shows the relationship among the cyclicity operators on a Hilbert space.
\[ \begin{array}{ccccccc}
\rm{ Hypercyclicity} &\Rightarrow& \rm{ Diskcyclicity} &\Rightarrow& \rm{Supercyclicity}.
  \end{array} \]

It was known that the hypercyclic operators are strictly infinite dimensional phenomena; however, the supercyclic operators exist on both one-dimensional and infinite dimensional Hilbert spaces \cite{Gerd}.\\

This article consists of three sections. The Section $2$ gives a brief review of some works on the diskcyclicity. Some characterizations of diskcyclic bilateral weighted shifts on $\LZ$ will be described. We give an example of a diskcyclic operator that is not hypercyclic and an example of a supercyclic operator that is not diskcyclic.\\

In Section $3$, we represent another equivalent version to the diskcyclicity criterion, and we use it to show that a unilateral backward weighted shift is hypercyclic if and only if it is diskcyclic. Through Proposition \ref{D on fin}, we prove the existence of diskcyclic operators on every one-dimensional complex Hilbert spaces ( A Hilbert space over the field of complex numbers) .
The diskcyclicity shares many structures with hypercyclicity and supercyclicity, nevertheless not all. Based on the previous works, an operator is hypercyclic (or supercyclic) if and only if its inverse is hypercyclic \cite{Kitai} ( or supercyclic respectively). However, through Example \ref{d. not adj} (Example \ref{d not inv} and Example \ref {inv not}), we show that the adjoint (inverse) of diskcyclic operators need not be diskcyclic. In addition, the somewhere density of the orbit of an operator(the cone generated by orbit) implies the everywhere density of the orbit (cone generated by orbit) \cite{Feldman}. However, we show that the somewhere density of the disk orbit of an operator does not imply to the everywhere density of it, by giving the Counterexample \ref{some not every}. Finally, in Corollary \ref{suf somwhere}, we give a sufficient condition for the somewhere density disk orbit of an operator to be everywhere dense and we also give some spectral properties of diskcyclic operators.
\section{Preliminaries.}
We denote the disk orbit $\{\al T^nx:n\geq 0, \al\in \co ,|\al|\leq 1\}$
by $\D Orb(T,x)$, the set of all diskcyclic vectors by $\D C(T)$ and the set of all diskcyclic operators on a Hilbert space $\h$ by $\D C(\h)$.\\
The following results are due to Zeana \cite{cyclic} unless otherwise stated.
%========================================================================================================================================================================
The following proposition gives a necessary but not sufficient condition for the diskcyclicity.
\begin{proposition}\label{}
If $x$ is a diskcyclic vector for $T$, then
$$ \inf\{\al\norm{T^nx}: n\ge 0, \al\in[0,1]\}=0
\mbox{ and } \sup\{\norm{T^nx}: n\ge 0\}=\iy.$$
\end{proposition}
\begin{proof}
Since $\al \in [0,1]$, then it is clear that  $\inf\{\al\norm{T^nx}: n\ge 0, \al\in[0,1]\}=0$. Towards a contradiction, assume that
$$
\sup\set{\al\norm{T^nx} : n\ge 0, \al\in[0,1]}=m<\infty,
$$
and $y\in\h$ such that  $\norm{y}> m$.
Since $T\in\D C(\h)$, then there exist sequences $\set{n_k}$ in $\N$ and $\set{\al_k}$ in $\co$; $\abs{\al_k}\le 1$ such that $\al_kT^{n_k}x\to y$. It follows that $\norm{y}\le m$  which is contradiction.
\end{proof}
%========================================================================================================================================================================
\begin{proposition}\label{direct}
If $\set{\h_i}_{i=1}^n$ is a family of Hilbert spaces, $T_i\in {\mathcal B(\mathcal H_i)}$ and $\oplus T_i\in \D C(\oplus\h_i)$, then $T_i\in \D C(\h_i)$ for all $1\le i \le n$.
\end{proposition}
\begin{proof}
Let $y=(y_1,y_2,\ldots)\in\oplus\h_i$ and $x=(x_1,x_2,\ldots)\in\D C(\oplus T_i)$, then there exist sequences
$\set{\al_k}$ in $\co$; $\abs{\al_k}\le 1$ and $\set{n_k}$ in $\N$ such that $\al_k(\oplus
T_i)^{n_k}x\to y$, as $n_k\to\infty$. It easily follows that $\al_kT_i^{n_k}x_i\to y_i$ for all $i$.
\end{proof}
%========================================================================================================================================================================
\begin{definition}
A bounded linear operator $T : \h \rightarrow \h $ is called {\bf disk transitive} if for any pair $U, V$ of nonempty open subsets of $X$, there exist $\alpha \in \co ; 0 < |\alpha| \leq 1$, and $n \geq 0$ such that $T^n(\alpha U) \cap V \neq \phi$ or equivalently, there exist $\alpha \in \co ; |\alpha| \geq 1$, and $n \geq 0$ such that $T^{-n}(\alpha U) \cap V \neq \phi$.
\end{definition}
%========================================================================================================================================================================
\begin{proposition}\label{1.3}
Let $\h$ and  ${\mathcal K}$ be Hilbert spaces, $T\in\BH$ and $S\in {\mathcal B(K)}$. Assume that $G:\h \rightarrow {\mathcal K}$ is a bounded linear transformation
with dense range and  $SG = GT$. If $T\in\D C(\h)$, then $S\in\D C({\mathcal K})$.
\[
  \begin{CD}
    \h @>{T}>> \h\\
    @V{G}VV @VV{G}V\\
    \mathcal{K} @>>{S}> \mathcal{K}
  \end{CD}
\]
\end{proposition}
\begin{proof}
Let $x\in \D C(T)$. Then we have
\begin{align*}
\overline{\D Orb(S,Gx)}&= \overline{\set{\al S^nGx:n\geq 0, \al\in \co ,|\al|\leq 1}}\\
                       &=\overline{\set{\al GT^nx:n\geq 0, \al\in \co ,|\al|\leq 1}}\\
											 &=\overline{G \set{\al T^nx:n\geq 0, \al\in \co ,|\al|\leq 1}}\\
											&\supseteq G\left( \overline{\set{\al T^nx:n\geq 0, \al\in \co ,|\al|\leq 1}}\right)\\
											&= G(\h).
\end{align*}
Since $R(G)$ is dense in ${\mathcal K}$, it follows that $\D Orb(S,Gx)$ is dense in ${\mathcal K}$. Thus $S\in \D C({\mathcal K})$ with diskcyclic vector $Gx$.
\end{proof}
%========================================================================================================================================================================
\begin{proposition} Let $ T,\, S\in \BH$ such that $ST=TS$ and $R(S)$ be a dense set in $\h$. If $x\in \D C(T)$, then $Sx\in \D C(T)$.
\end{proposition}
\begin{proof}
Since $x\in \D C(T)$, then
\begin{align*}
\overline{\D Orb(T,Sx)}&= \overline{\set{\al T^nSx:n\geq 0, \al\in \co ,|\al|\leq 1}}\\
                       &=\overline{\set{\al ST^nx:n\geq 0, \al\in \co ,|\al|\leq 1}}\\
											 &=\overline{S \set{\al T^nx:n\geq 0, \al\in \co ,|\al|\leq 1}}\\
											&\supseteq S\left( \overline{\set{\al T^nx:n\geq 0, \al\in \co ,|\al|\leq 1}}\right)\\
											&= S(\h)=R(S).
\end{align*}
Thus $\D Orb(T,Sx)$ is dense in $\h$ and hence $Sx\in \D C(T)$.
\end{proof}
%========================================================================================================================================================================
From the last proposition one can easily deduce that there are many diskcyclic vectors if the operator has one diskcyclic vector.
\begin{corollary}\label{huge}
If $x$ is a diskcyclic vector for $T$, then $T^nx$ is also a diskcyclic vector for $T$ for all $ n\in \N$.
\end{corollary}
%========================================================================================================================================================================
\begin{proposition}\label{DT}
Every diskcyclic operator on $\h$ is disk transitive.
\end{proposition}
\begin{proof}
Let $T$ be a diskcyclic operator. Then, by the previous corollary, it is clear that $\D C(T)$ is a dense set. Assume that $U$ and $V$ are two open sets. Then there exist an $\al \in \co;\, |\al|\le 1$ and a non-negative integer $N$ such that $\al T^Nx\in U$. Also one can find $\lm \in \co,\, |\lm|\le |\al| $ and $n \ge N,$ such that $\lm T^nx\in V$. Thus $(\lm/\al)T^{n-N}U\cap V\neq \phi$.
\end{proof}
%========================================================================================================================================================================
\begin{proposition} \label{cap}
Every disk transitive operator is diskcyclic and
$$\displaystyle{\D C(T)=\bigcap_k\left(\bigcup_{\stackrel{\lm \in \co}{|\lm|\geq1}}\bigcup_nT^{-n}(\lm B_k)\right)}$$
 is a dense $G_\delta $ set, where $\{B_k\}$ is a countable open basis for $\h$.
\end{proposition}
\begin{proof}
We have $x \in \D C(T)$ if and only if $\{\alpha T^nx:n\geq 0, |\alpha|\leq 1\}$ is dense in $\h$ if and only if for each $k>0$, there exist $\alpha \in \co; |\alpha|\leq 1,$ and $n \in \N$ such that $\alpha T^nx \in B_k$ if and only if $\displaystyle{x\in \bigcap_k\left(\bigcup_{\stackrel{\lm \in \co}{|\lm|\geq1}}\bigcup_nT^{-n}(\lm B_k)\right)}$.\\if and only if
$$\displaystyle{\D C(T)=\bigcap_k\left(\bigcup_{\stackrel{\lm\in \co}{|\lm|\geq 1}}\bigcup_nT^{-n}(\lm B_k)\right)}.$$
Since $\D C(T)$ can be written as a countable intersection of open sets, then $\D C(T)$ is a $G_\delta$ set. Moreover, it follows from the Baire category theorem that $\D C(T)$ is dense if and only if  each open set $A_k=\bigcup_{\stackrel{\lm \in \co}{|\lm|\geq1}}\bigcup_nT^{-n}(\lm B_k)$ is dense; i.e, if and only if for each non-empty open set $U$ and any $k\in \N$ one can find $n$ and $\lm\in \co; |\lm|\ge 1$ such that
$$U\cap T^{-n}(\lm B_k)\neq\phi \quad \mbox{i.e} \quad \frac{1}{\lm}T^nU\cap B_k\neq\phi$$
Since $\{B_k\}$ is a countable open basis for $\h$, this is equivalent to the disk transitivity of $T$.
\end{proof}
%=======================================================================================================================
\begin{corollary}
Every vector in $\h$ can be written as a sum of two diskcyclic vectors for a diskcyclic operator $T$.
\end{corollary}
%==============================================================================================================================================================
The following proposition gives some equivalent assertions to diskcyclicity.
\begin{proposition}\label{d.t.}
Let $T\in\BH$. The following statements are equivalent.
\begin{enumerate}
  \item \label{a1} $T\in\D C(\h)$.
	\item \label{b1} $T$ is disk transitive.
  \item \label{c1} For each $x,y\in\h$, there exist sequences $\{x_k\}$ in $\h$, $\{n_k\}$ in $\N$, and $\{\al_k\}$ in $\co$;
  $0<\abs{\al_k}\le 1$ such that $x_k\to x$ and $T^{n_k}\al_kx_k \to y$.
  \item \label{d1} For each $x,y\in\h$ and each neighborhood $W$ of
  the zero in $\h$, there exist $z\in\h$, $n\in\N$, and $\al\in\co$;
  $0<\abs{\al}\le 1$ such that $x-z\in W$ and $T^n\al z-y\in W$.
\end{enumerate}
\end{proposition}
\begin{proof}
\ref{a1} $\Leftrightarrow$ \ref{b1}: Follow from Proposition \ref{DT} and Proposition \ref{cap}.\\

\ref{b1} $\Rightarrow $ \ref{c1}: Let $x,y\in\h$, and let
 $B_k=\mathbb B(x,1/k)$, $B'_k=\mathbb B(y,1/k)$ for all $k\geq 1$. From part (\ref{b1})
there exist sequences $\set{n_k}$ in $\N$, $\set{\al_k}$ in $\co$;
$0<\abs{\al_k}\le1$ for all $k\ge1$ and $\set{x_k}$ in $\h$ such that
$x_k\in B_k$ and $T^{n_k}\al_k x_k\in B'_k$ for
all $k\ge 1$. Then $\norm{x_k-x}<1/k$ and $\norm{T^{n_k}\al_k
x_k-y}<1/k$ for all $k\ge1$.\\

\ref{c1} $\Rightarrow$ \ref{d1}: Follows immediately from part (\ref{c1}) by taking $z=x_k$ for a large enough $k\in \N$.\\

\ref{d1} $\Rightarrow$ \ref{b1}: Let $U$ and $V$ be two non-empty open subset of $\h$. Let $W$ be a neighborhood for zero, pick $x\in U$ and $y\in V$, so there exist $z\in\h$, $n\in\N$, $\al\in\co, \,|\al|\le 1$  such that $x-z\in W$ and $T^n\al z-y\in W$. It follows immediately that $z\in U$ and $T^n\al z\in V.$
 \end{proof}
%========================================================================================================================================================================
\begin{proposition}[ Diskcyclicity criterion]\label{dc}
Let $T\in\BH$ with the following properties.
\begin{enumerate}
  \item \label{a6} There exist two dense sets $X$, $Y$ in $\h$ and right inverse
  of $T$ (not necessary bounded) $S$ such that  $S(Y)\subseteq Y$ and
  $TS=I_{Y}$.
  \item \label{b6} There is a sequence $\{n_k\}$ in $\N$ such that
\begin{enumerate}
 \item  $\lim_{k\to\infty}\norm{S^{n_k}y}=0$ for all
  $y\in Y$;
  \item $\lim_{k\to\infty}\norm{T^{n_k}x}\norm{S^{n_k}y}=0$ for all
  $x\in X$, $y\in Y$.
  \end{enumerate}
\end{enumerate}
Then $T\in\D C(\h)$.
\end{proposition}
\begin{proof} The proof follows the lines of the proof of \cite[Theorem 1.6]{dynamic}  and  \cite[Theorem 1.14]{dynamic}; therefore, we will skip it.
\end{proof}
%========================================================================================================================================================================
The characterizations of the bilateral weighted shifts that are diskcyclic are shown in the following results.
\begin{theorem} \label{4.3.1}
Let $T$ be a bilateral forward weighted shift on the Hilbert space $\h=\LZ$ with the weight sequence
$\set{w_n}_{n\in\Z}$. Then the following statements are
equivalent.
\begin{enumerate}
  \item \label{a4} $T\in\D C(\h)$.
  \item \label{b4} For all $q\in\N$,
\begin{enumerate}
  \item  \label{d4} $\displaystyle\limsup_{n \to \infty}\min_q\set{\prod_{k=h-n}^{h-1}w_k:\abs{h}\le
  q}=\infty$.
  \item  $\displaystyle\liminf_{n\to\infty}\max_q\set{\frac{\prod_{k=j}^{j+n-1}w_k}
      {\prod_{k=h-n}^{h-1}w_k}:\abs{h},\abs{j}\le
  q}=0$.
\end{enumerate}
  \item \label{c4} $T$ satisfies the diskcyclicity criterion.
\end{enumerate}
\end{theorem}
\begin{proof}
\ref{a4} $\Rightarrow$ \ref{b4}: The proof is similar to \cite[Theorem 3.1]{Sup & weig} and observe that the condition (\ref{d4})  holds from the fact that $|\al|\le1$.\\
\ref{b4} $\Rightarrow$ \ref{c4}: Let $X=Y$ be the manifold spanned by $\set{e_n}_{n\in \Z}$ and let $x=\sum_{|j|\le q}x_je_j$ and $y=\sum_{|j|\le q}y_je_j$. Assume that $B$ is the right inverse of $T$, then
$$Be_n=\frac{1}{w_{n-1}}e_{n-1},$$
 and
$$\norm{T^nx}\le\max_{q}\set{\prod_{k=j}^{j+n-1}w_k:\abs{j}\le q}\norm{x};$$
$$\norm{B^ny}\le\frac{\norm{y}}{\min_{q}\set{\prod_{k=h-n}^{h-1}w_k:\abs{h}\le q}}.$$
Thus
\[\norm{T^nx}\norm{B^ny}\le\max_{q}\set{\frac{\prod^{j+n-1}_{k=j}w_k}{\prod_{k=h-n}^{h-1}w_k}:\abs{j},\abs{h}\le q}\norm{x}\norm{y}.\]
Let $\eps>0$ and $q\in \N$. Assume there exists a positive integer $n>2q$ satisfies $\left({\prod^{j+n-1}_{k=j}w_k}\right)/\left({\prod_{k=h-n}^{h-1}w_k}\right)<\eps$  and $\prod_{k=h-n}^{h-1}w_k> {1}/{\eps}$ for all $|j|,|h|\le q$. Then
\[\lim_{n\to\infty}\norm{T^{n}x}\norm{B^{n}y}=0\qquad\text{and}\qquad
\lim_{n\to\infty}\norm{B^{n}y}=0.\]
\ref{c4} $\Rightarrow$ \ref{a4}: Follows from Proposition \ref{dc}.
\end{proof}
%========================================================================================================================================================================
\begin{corollary}\label{4.3.2}
Let $T\in \D C(\LZ)$ be a forward weighted shift with weight
sequence $\set{w_n}_{n\in\Z}$. Then there is a sequence
$\set{n_r}$ in $\N$ such that
\begin{enumerate}
  \item $\displaystyle\lim_{r\to\infty}\brc{\prod_{k=1}^{n_r}\frac{1}{w_{-k}}}=0$,
  \item $\displaystyle\lim_{r\to\infty}\brc{\prod_{k=1}^{n_r}w_{k}}
  \brc{\prod_{k=1}^{n_r}\frac{1}{w_{-k}}}=0$.
\end{enumerate}
\end{corollary}
\begin{proof}
By Theorem \ref{4.3.1}, take $q=0$, then $j=h=0$. Hence
\[\limsup_{n\to\infty}\brc{\prod_{k=-n}^{-1}w_k}=\infty\] and
\[\liminf_{n\to\infty}\brc{\prod_{k=0}^{n-1}w_k}\brc{\prod_{k=1}^{n}\frac{1}{w_{-k}}}=0.\]
Thus, there is a sequence $\set{n_r}$ in $\N$ such that
\[\lim_{r\to\infty}\brc{\prod_{k=1}^{n_r}\frac{1}{w_{-k}}}=0\] and
\[\lim_{r\to\infty}\brc{\prod_{k=0}^{n_r-1}w_{k}}\brc{\prod_{k=1}^{n_r}\frac{1}{w_{-k}}}=0.\]
Let $\eps>0$, then there exists a positive integer $m>0$ such that for all $r>m$
\[\abs{\brc{\prod_{k=0}^{n_r-1}w_{k}}\brc{\prod_{k=1}^{n_r}
\frac{1}{w_{-k}}}}<\eps\abs{\frac{w_0}{w_{n_r}}}.\] Hence
\[\abs{\brc{\prod_{k=1}^{n_r}w_{k}}\brc{\prod_{k=1}^{n_r}\frac{1}{w_{-k}}}}=
\abs{\brc{\prod_{k=0}^{n_r-1}w_{k}}\brc{\prod_{k=1}^{n_r}\frac{1}{w_{-k}}}}
 \abs{\frac{w_{n_r}}{w_0}}<\eps.\]
Therefore
\[\lim_{r\to\infty}\brc{\prod_{k=1}^{n_r}\frac{1}{w_{-k}}}=0\quad\text{and}\quad
\lim_{r\to\infty}\brc{\prod_{k=1}^{n_r}w_k}
  \brc{\prod_{k=1}^{n_r}\frac{1}{w_{-k}}}=0.\]
\end{proof}
%========================================================================================================================================================================
\begin{proposition}\label{4.3.4}
Let $T:\LZ\to\LZ$ be an invertible forward weighted shift with
weight sequence $\set{w_n}_{n\in\Z}$. Then $T\in\D C(\LZ)$ if and only if
there exists a sequence $\set{n_r}$ in $\N$ such that
\begin{enumerate}
  \item $\displaystyle\lim_{r\to\infty}\prod_{k=1}^{n_r}\frac{1}{w_{-k}}=0$;
  \item   $\displaystyle\lim_{r\to\infty}\brc{\prod_{k=1}^{n_r}w_{k}}
 \brc{\prod_{k=1}^{n_r}\frac{1}{w_{-k}}}=0$.
\end{enumerate}
\end{proposition}
\begin{proof}
The first part follows from Corollary \ref{4.3.2}.\\
Conversely, we will verify the
diskcyclicity criterion. Let \[X=Y=\set{x\in\LZ: x\,\,
\text{has only finitely many non--zero coordinates}},\] and let $B$ be the inverse of $T$.
It is sufficient, by linearity, the triangle inequality, \cite[Lemma 3.1]{H C Bi.} and \cite[Lemma 3.3]{H C Bi.} to suppose that $x=e_1$ and $y=e_0$.
Since $\lim_{r\to\infty}\norm{B^{n_r}e_{0}}=\lim_{r\to\infty}\prod_{k=1}^{n_r}\left({1}/{w_{-k}}\right)=0$
and \[\lim_{r\to\infty}\norm{T^{n_r}e_{1}}\norm{B^{n_r}e_{0}}=
\lim_{r\to\infty}\brc{\prod_{k=1}^{n_r}w_{k}}\brc{\prod_{k=1}^{n_r}\frac{1}{w_{-k}}}=0,\]
Then $T$ is diskcyclic.
\end{proof}
%========================================================================================================================================================================
 We should note that a bilateral weighted shift operator is invertible if and only if there is a positive real number $m$ such that $|w_n|\geq m$ for all $n\in \Z$.
%========================================================================================================================================================================
The following corollary shows that Proposition \ref{4.3.4} still holds for some further general cases.
\begin{corollary}\label{bif d}
Let $T:\LZ\to\LZ$ be a bilateral forward weighted shift with the weight
sequence $\set{w_n}_{n\in\Z}$ and assume that there exists a positive integer $m > 0$ such that $w_n \ge m$ for all $n < 0$ (or for all $n > 0$). Then $T\in\D C(\LZ)$ if and only if there
is a sequence  $\set{n_r}$ in $\N$ such that
\begin{enumerate}
  \item \label{a2}  $\displaystyle\lim_{r\to\infty}\prod_{k=1}^{n_r}\frac{1}{w_{-k}}=0$;
  \item  \label{b2} $\displaystyle\lim_{r\to\infty}\brc{\prod_{k=1}^{n_r}w_{k}}\brc{\prod_{k=1}^{n_r}\frac{1}{w_{-k}}}=0$.
\end{enumerate}
\end{corollary}
\begin{proof}
The  proof of ``if " part follows from Theorem \ref{4.3.1} and it is similar to the proof of \cite[Theorem 4.1]{H C Bi.}; therefore, we will skip the proof.\\
The ``only if " part follows from Corollary \ref{4.3.2}.
\end{proof}
%========================================================================================================================================================================
Since the bilateral weighted backward shifts are unitarily equivalent to the  bilateral weighted forward shifts, then the above results are extended to backward shift operators and their proofs can be proved by similar steps.
%========================================================================================================================================================================
\begin{theorem} \label{1}
Let $T$ be a bilateral backward weighted shift on the Hilbert space $\h=\LZ$ with the  weight sequence
$\set{w_n}_{n\in\Z}$. Then the following statements are
equivalent.
\begin{enumerate}
  \item $T\in\D C(\h)$.
  \item For all $q\in\N$;
\begin{enumerate}
  \item  $\displaystyle\limsup_{n \to \infty}\min_q\set{\prod_{k=h+1}^{h+n}w_k:\abs{h}\le
  q}=\infty$;
  \item  $\displaystyle\liminf_{n\to\infty}\max_q\set{\frac{\prod_{k=j+1-n}^j w_k}{\prod_{k=h+1}^{h+n}w_k}:\abs{h},\abs{j}\le
  q}=0$.
\end{enumerate}
  \item $T$ satisfies the diskcyclicity criterion.
\end{enumerate}
\end{theorem}
%========================================================================================================================================================================
\begin{corollary}\label{2}
Let $T\in \D C(\LZ)$ be a  backward weighted shift with the weight
sequence $\set{w_n}_{n\in\Z}$. Then there is a sequence
$\set{n_r}$ in $\N$ such that
\begin{enumerate}
  \item $\displaystyle\lim_{r\to\infty}\brc{\prod_{k=1}^{n_r}\frac{1}{w_k}}=0$;
  \item $\displaystyle\lim_{r\to\infty}\brc{\prod_{k=1}^{n_r}w_{-k}}
  \brc{\prod_{k=1}^{n_r}\frac{1}{w_{k}}}=0$.
\end{enumerate}
\end{corollary}
%========================================================================================================================================================================
\begin{proposition}\label{3}
Let $T:\LZ\to\LZ$ be an invertible  backward weighted shift with
the weight sequence $\set{w_n}_{n\in\Z}$. Then $T\in\D C(\LZ)$ if and only if
there exists a sequence $\set{n_r}$ in $\N$ such that
\begin{enumerate}
  \item $\displaystyle\lim_{r\to\infty}\prod_{k=1}^{n_r}\frac{1}{w_{k}}=0$;
  \item   $\displaystyle\lim_{r\to\infty}\brc{\prod_{k=1}^{n_r}w_{-k}}
 \brc{\prod_{k=1}^{n_r}\frac{1}{w_{k}}}=0$.
\end{enumerate}
\end{proposition}
%========================================================================================================================================================================
\begin{corollary}\label{biB c}
Suppose $T:\LZ\to\LZ$ is a bilateral  backward weighted shift with the weight
sequence $\set{w_n}_{n\in\Z}$ and there exists  a positive integer $m > 0$ such that $w_n \ge m$ for all $n < 0$ (or for all $n > 0$). Then $T\in\D C(\LZ)$ if and only if there
is a sequence  $\set{n_r}$ in $\N$ such that
\begin{enumerate}
  \item   $\displaystyle\lim_{r\to\infty}\prod_{k=1}^{n_r}\frac{1}{w_{k}}=0$;
  \item   $\displaystyle\lim_{r\to\infty}\brc{\prod_{k=1}^{n_r}w_{-k}}\brc{\prod_{k=1}^{n_r}\frac{1}{w_{k}}}=0$.
\end{enumerate}
\end{corollary}
%========================================================================================================================================================================
Now, we will give an example of a diskcyclic operator which is not hypercyclic.
\begin{example}\label{2.12}
Let $T:\LZ\to\LZ$ be the bilateral forward weighted shift with
the weight sequence
\begin{equation*}
w_n=
\begin{cases} 2 & \text{if $n \geq 0$,}
\\
3 &\text{if $n< 0$.}
\end{cases}
\end{equation*}
Then $T$ is diskcyclic but not hypercyclic.
\end{example}
\begin{proof}
By applying Corollary \ref{bif d} and taking $n_r=n$ (the set of all natural numbers). Observe that $$\displaystyle\lim_{n\to\infty}\prod_{k=1}^{n}\frac{1}{w_{-k}} = \displaystyle\lim_{n\to\infty}\prod_{k=1}^{n}\frac{1}{3}=\displaystyle\lim_{n\to\infty}\frac{1}{3^n}=0;$$
and
$$\displaystyle\lim_{n\to\infty}\brc{\prod_{k=1}^{n}w_{k}}\brc{\prod_{k=1}^{n}\frac{1}{w_{-k}}}= \displaystyle\lim_{n\to\infty}\brc{\prod_{k=1}^{n}2}\brc{\prod_{k=1}^{n}\frac{1}{3}}=\displaystyle\lim_{n\to\infty}(2^n)(\frac{1}{3^n})=0.$$
Thus by Corollary  \ref{bif d}, $T$ is diskcyclic. On the other hand, since for all increasing sequence ${n_r}$ of positive integers
$$\displaystyle\lim_{r\to\infty}\brc{\prod_{k=1}^{n_r}w_{k}}=\displaystyle\lim_{r\to\infty}
\brc{\prod_{k=1}^{n_r}2}=\displaystyle\lim_{n\to\infty}(2^{n_r})=\iy,$$
from \cite[Theorem 4.1.]{H C Bi.}, $T$ can not be hypercyclic.
\end{proof}
%========================================================================================================================================================================
The following example gives us an operator which is supercyclic but not diskcyclic.
\begin{example}
Let $T:\LZ\to\LZ$ be a bilateral forward weighted shift with the
weight sequence
\begin{equation*}
w_n=
\begin{cases} \frac{1}{3} & \text{if $n \geq 0$,}
\\
\frac{1}{2} &\text{if $n< 0$.}
\end{cases}
\end{equation*}
Then $T$ is supercyclic but not diskcyclic.
\end{example}
\begin{proof}
By taking $n_r=n$, we have

$$\displaystyle\lim_{n\to\infty}\brc{\prod_{k=1}^{n}w_{k}}\brc{\prod_{k=1}^{n}\frac{1}{w_{-k}}}= \displaystyle\lim_{n\to\infty}\brc{\prod_{k=1}^{n}\frac{1}{3}}\brc{\prod_{k=1}^{n}2}=
\displaystyle\lim_{n\to\infty}(\frac{1}{3^n})(2^n)=0.$$
From \cite[Theorem 4.1]{H C Bi.}, $T$ is supercyclic. However, for all increasing sequence ${n_r}$ of positive integers we have
$$\displaystyle\lim_{r\to\infty}\prod_{k=1}^{n_r}\frac{1}{w_{-k}} = \displaystyle\lim_{r\to\infty}\prod_{k=1}^{n_r}2=\displaystyle\lim_{r\to\infty}2^{n_r}=\iy.$$
Hence, by Corollary \ref{bif d}, $T$ is not diskcyclic.
\end{proof}
%========================================================================================================================================================================
The following results are noteworthy spectral properties of diskcyclic operators.
\begin{proposition} \label{2.16}
Let $T\in\D C(\h)$. Then $T^*$ has at most one eigenvalue and its modules is greater than $1$.
\end{proposition}
\begin{proof}
Since $T\in SC(\h)$, then $\si_p(T^*)$ contains at most one
non-zero eigenvalue; say $\lm$ \cite[Proposition 3.1.]{limits}. Hence, there is a unit vector
$z\in\h$ such that  $T^*z=\lm z$. Let $x\in\D C(T)$. Then, it is easy to prove
that
\begin{equation}\label{144}
\set{\abs{\iner{\mu T^nx,z}}\suc n\ge0,\mu\in
\co;\abs{\mu}\le1}\,\text{is dense in}\,\R^+\cup\set{0}.
\end{equation}
Note that for all $n\ge1$, $\abs{\iner{\mu
T^nx,z}}\le\abs{\mu}^n\abs{\lm}^n\abs{\iner{x,z}}$. If we suppose $\abs{\lm}\le1$,
then \[\abs{\iner{\mu T^nx,z}}\le\abs{\iner{x,z}},\]
which contradicts (\ref{144}). Therefore, we reach the desired result.
\end{proof}
%========================================================================================================================================================================
\begin{corollary}\label{2.17}
Let $T\in\BH$. If $\sigma(T)$ has a connected component $\sigma$ such that
$\sigma\subset B(0,1)$, then $T\not\in\D C(\h)$.
\end{corollary}
\begin{proof}
Towards a contradiction, suppose that a diskcyclic operator $T$ has a connected component $\sigma$ such that $\sigma\subset B(0,1)$. Then, by Riesz decomposition Theorem, $T=T_1\oplus T_2$ such that $\sigma(T_1)=\sigma$. It follows that $\D Orb(T_1,x)$ is bounded for all $x\in \h$ and hence $T_1$ can not be dense in $\h$, a contradiction to Proposition \ref{direct}.
\end{proof}
%========================================================================================================================================================================
%########################################################################################################################################################################
\section{Main Results}
%########################################################################################################################################################################
We adjust the diskcyclicity criterion in order to obtain another version.
\begin{theorem}[Second Diskcyclicity Criterion]\label{dc2}
Let $T \in \BH$. If there exists an increasing sequence of integers $\{n_k\}_{k\in \N}$ and two dense sets $D_1, D_2 \in \h$ such that
\begin{enumerate}[(a)]
\item \label{a7} For each $y\in D_2$, there exists a sequence $\{x_k\}$ in $\h$ such that  $x_k \rightarrow 0$, and $T^{n_k}x_k \rightarrow y$,
\item \label{b7} $\left\|T^{n_k}x\right\|\left\|x_k\right\|\rightarrow 0$  for all $x \in D_1$,
\end{enumerate}
then $T$ is diskcyclic.
\end{theorem}
\begin{proof}
We will verify Proposition \ref{dc}. Since for each $y\in D_2$ there exists a sequence $\{x_k\}$ in $\h$ such that $T^{n_k}x_k \rightarrow y$, then there are maps $S$ which are right inverses of $T$ on $D_2$ such that $x_k=S^{n_k}y$ for some increasing sequence of integers $\{n_k\}_{k\in \N}$. Moreover, the condition (\ref{b6}) of Proposition \ref{dc} directly follows.
\end{proof}
%================================================================================================================================================
\begin{proposition}
 Both the diskcyclicity criteria are equivalent to each other.
\end{proposition}
\begin{proof}
If $T$ satisfies the diskcyclicity criterion then, by setting $S^{n_k}y=x_k$, we get the desired result.\\
The other implication follows from Theorem \ref{dc2}.
\end{proof}
%================================================================================================================================================
Now we illustrate the Second Diskcyclicity Criterion with the following example.
\begin{example}\label{2B}
Let $T=2B$, where $B$ is the unilateral backward shift on $\LN$. Then $T$ is diskcyclic.
\end{example}
\begin{proof}
Let $X=Y$ be the dense set in $\LN$ such that all except finitely many coordinates of each element of $X$ are zero and let $n_k = k$ (the set of all natural numbers). Then, we will achieve conditions (\ref{a7}) and (\ref{b7}) of Theorem \ref{dc2}.
For each $y\in Y$ assume that $x_k=\left(({1}/{2})F\right)^ky$, where $F$ is the unilateral forward shift on $\LN$. Therefore, $\norm{x_k}=\norm{\left({1}/{2^k}\right)F^ky}\to 0$ as $k\to \iy$. Moreover, $T^kx_k=(2B)^k({F}/{2})^ky=y$. \\
Now, since $\norm{B^kx}=0$ eventually for a large enough $k$, then we have
$$\norm{T^kx}\norm{x_k}=\norm{2^kB^kx}\left\|{\frac{1}{2^k}F^ky}\right\|=\norm{B^kx}\norm{F^ky}=0.$$
It follows that $T$ satisfies the Second Diskcyclicity Criterion and so $T$ is a diskcyclic operator.
\end{proof}
%================================================================================================================================================
In general case, we have the following example.
\begin{example}\label{a>B}
Let $T=aB$, where $B$ is the unilateral backward shift on $\LN$ and $|a| >1$. Then $T$ satisfies the Second Diskcyclicity Criterion.
\end{example}
%================================================================================================================================================
The next example shows that not every multiple of the unilateral backward shift is diskcyclic.
\begin{example}\label{a<B}
If $T=aB$, where $|a|\leq 1$, and $B$ is the unilateral backward shift on $\LN$, then $T$ is not diskcyclic.
\end{example}
\begin{proof}
Since $T^k(e_n)={a^k}e_{n-k}$, where $\{e_n\}_{n\in \N}$ is the canonical basis of $\LN$, then $\left\|T^k(e_n)\right\|={|a|^k}\le 1$. It follows that $\alpha \left\|T^k(x)\right\| \to 0$ for all $x\in \LN$ and $\alpha\in\co,\, \left|\alpha\right|\leq 1$. Thus $\D Orb(T,x)$ is bounded and can not be dense in $\LN$.
\end{proof}
%================================================================================================================================================
From Example \ref{a>B}, Example \ref{a<B}  and \cite[Corollary 1.6]{Kitai} we can easily deduce the following corollary.
\begin{corollary}\label{Back not D}
A multiple of a unilateral backward shift on $\LN$ is hypercylcic if and only if it is diskcyclic.
\end{corollary}
%================================================================================================================================================
The following corollary follows directly from \cite[Theorem 1.40]{dynamic} and Corollary \ref{Back not D}.
\begin{corollary}
If $B$ is a unilateral backward weighted shift with weight sequence $\set{w_n}$, then $B$ is diskcyclic if and only if $$\lim \sup_{n\to \iy}\brc{w_1,w_2\ldots w_n}=\iy.$$
\end{corollary}
%================================================================================================================================================
The following example shows that the diskcyclic operators exist on the complex line $\co$
\begin{example}\label{diskC}
Let $T\in B(\co)$ defined as $T(x)=2x$, then $T$ is diskcyclic on $\co$
\begin{proof}
We have $\D Orb(T,x)=\{\alpha 2^n x: n\geq 0$ and $ |\alpha|\leq 1\}$. Assume that $x=1$, then $\D Orb(T,1)=\{\alpha 2^n : n\geq 0$ and $ |\alpha|\leq 1\}$. Let $z=a+bi \in \co$ where $a,b \in \R$, then choose $k \in \N$ in which $2^k \geq \sqrt{|a|^2+|b|^2}$. It follows that $z=2^k\left(\frac{a}{2^k}+\frac{b}{2^k}i \right)\in \D Orb(T,1)$. Thus $T$ is diskcyclic operator.
\end{proof}
\end{example}
\begin{remark}\label{iso}
Every two $n$-dimensional Hilbert spaces over the field of complex numbers are isomorphic.
\end{remark}
The following proposition shows the existence of diskcyclic operators on every one dimensional complex Hilbert space.
\begin{proposition}\label{D on fin}
There exist a diskcyclic operator $T$ on a nontrivial complex Hilbert space $\h$ if and only if $dim(\h)=1$ or $dim(\h)=\iy$
\begin{proof}
Assume that $T$ is diskcyclic operator, then $T$ is supercyclic operator and hence by \cite{Hilden}, $dim(\h)=1$ or $dim(\h)=\iy$.\\
Conversely, follow from Example \ref{diskC}, Remark \ref{iso} and \cite{12}.
\end{proof}
\end{proposition}
It can be easily checked that the adjoint of the operator in Example \ref{diskC} is diskcyclic. However, this property is not true in general as the following example.
\begin{example}\label{d. not adj}
Let $T$ be defined on $\LZ$ as in Example \ref{2.12}. Then $T$ is diskcyclic but its adjoint is not.
\end{example}
\begin{proof}
From Example \ref{2.12} we have that $T$ is diskcyclic. Now, we will show that $T^*$ is not diskcyclic. It is clear that $T^*e_n=Be_n$, where $B$ is the bilateral backward weighted shift with weight sequence
\begin{equation*}
z_n=
\begin{cases} 2 & \text{if $n > 0$,}
\\
3 &\text{if $n \leq 0$.}
\end{cases}
\end{equation*}
Then for all increasing sequence $n_r$ of positive integers, we have
   $$\displaystyle\lim_{r\to\infty}\prod_{k=1}^{n_r}\frac{1}{z_{k}}=\displaystyle\lim_{r\to\infty}\prod_{k=1}^{n_r}\frac{1}{2}=\displaystyle\lim_{n_r\to\infty}\frac{1}{2^{n_r}} =0,$$ however,
   $$\displaystyle\lim_{r\to\infty}\brc{\prod_{k=1}^{n_r}z_{-k}}\brc{\prod_{k=1}^{n_r}\frac{1}{z_{k}}}=
   \displaystyle\lim_{r\to\infty}\brc{\prod_{k=1}^{n_r}3}\brc{\prod_{k=1}^{n_r}\frac{1}{2}}=
   \displaystyle\lim_{r\to\infty}\left(\frac{3}{2}\right)^{n_r}=\iy.$$
From Corollary \ref{biB c}, it follows that  $T^*$ is not diskcyclic.
\end{proof}
%================================================================================================================================================
It has been shown that an  invertible $T\in \BH$ is hypercyclic(or supercyclic) if and only if $T^{-1}$ is hypercyclic(or supercyclic respectively). However, this equivalence is not necessarily true for the diskcyclicity case.
\begin{example}\label{d not inv}
Let $T$ be defined on $\co$ as in Example \ref{diskC}. then $T^{-1}$ is not diskcyclic.
\end{example}
\begin{proof}
Since $T^{-1}x=\left({1}/{2}\right)x$, then $\D Orb(T^{-1},y)$ is bounded for all $y\in \co$ and hence can not be dense in $\co$. It follows that $T^{-1}$ is not a diskcyclic operator.
\end{proof}
%================================================================================================================================================
Here, we recall that in the infinite dimensional spaces there are also diskcyclic operators whose inverses are not diskcyclic.
\begin{example}\label {inv not}
Let $T$ be defined on $\LZ$ as in Example \ref{2.12}. Then $T$ is diskcyclic, but its inverse is not.
\end{example}
\begin{proof}
Since $|w_n|\geq 2$ for all $n\in \Z$, then $T$ is invertible. The inverse of $T$ is the bilateral backward weighted shift with the weight sequence
\begin{equation*}
z_n=\frac{1}{w_{n-1}}=
\begin{cases} \frac{1}{2} & \text{if $n > 0$,}
\\
\frac{1}{3} &\text{if $n \leq 0$.}
\end{cases}
\end{equation*}
For all increasing sequence $n_r$ of positive integers, we have
$$\displaystyle\lim_{r\to\infty}\prod_{k=1}^{n_r}\frac{1}{z_{k}}=\displaystyle\lim_{r\to\infty}\prod_{k=1}^{n_r}2=
\displaystyle\lim_{r\to\infty}{2^{n_r}} =\iy.$$
By Corollary \ref{biB c}, we have that $T^{-1}$ is not a diskcyclic operator.
\end{proof}
%================================================================================================================================================
There are diskcyclic operators, in which their inverses are also diskcyclic as in the next example.
\begin{example}
Let $F$ be a bilateral forward weighted shift on $\LZ$ with weight sequence
\begin{equation*}
w_n=
\begin{cases} \frac{1}{2} & \text{if $n \geq 0$,}
\\
3 &\text{if $n< 0$.}
\end{cases}
\end{equation*}
Then both $F$ and $F^{-1}$ are diskcyclic.
\end{example}
\begin{proof}
Since $|w_n|\geq {1}/{2}$ for all $n\in \Z$, then $F$ is invertible. Also we have
   $$\displaystyle\lim_{n\to\infty}\prod_{k=1}^{n}\frac{1}{w_{-k}}=\displaystyle\lim_{n\to\infty}
   \prod_{k=1}^{n}\frac{1}{3}=\displaystyle\lim_{n\to\infty}\frac{1}{3^n}=0,$$ and
   $$\displaystyle\lim_{n\to\infty}\brc{\prod_{k=1}^{n}w_{k}}\brc{\prod_{k=1}^{n}
   \frac{1}{w_{-k}}}=\displaystyle\lim_{n\to\infty}\brc{\prod_{k=1}^{n}\frac{1}{2}}
   \brc{\prod_{k=1}^{n}\frac{1}{3}}=\displaystyle\lim_{n\to\infty}\frac{1}{2^n3^n}=0.$$\\
It follows from Corollary \ref{bif d} that $F$ is diskcyclic. Moreover, the inverse of $F$ is the bilateral backward weighted shift $Be_n=\left({1}/{w_{n-1}}\right)e_{n-1}$ with weight sequence
\begin{equation*}
z_n=\frac{1}{w_{n-1}}=
\begin{cases} 2 & \text{if $n > 0$,}
\\
\frac{1}{3} &\text{if $n \leq 0$.}
\end{cases}
\end{equation*}

 Since
   $$\displaystyle\lim_{n\to\infty}\prod_{k=1}^{n}\frac{1}{z_{k}}=\displaystyle\lim_{n\to\infty}
   \prod_{k=1}^{n}\frac{1}{2}=\displaystyle\lim_{n\to\infty}\frac{1}{2^n} =0,$$ and
   $$\displaystyle\lim_{n\to\infty}\brc{\prod_{k=1}^{n}z_{-k}}\brc{\prod_{k=1}^{n}\frac{1}{z_{k}}}=
   \displaystyle\lim_{n\to\infty}\brc{\prod_{k=1}^{n}\frac{1}{3}}\brc{\prod_{k=1}^{n}\frac{1}{2}}=
   \displaystyle\lim_{n\to\infty}\frac{1}{2^n3^n}=0,$$

by Corollary \ref{biB c}, we have that $F^{-1}$ is diskcyclic.
 \end{proof}
%================================================================================================================================================
Bourdon and Feldman in \cite{Feldman} proved that if $Orb(T,x) \brc{\text{or}\, \co Orb\brc{T,x}}$ is somewhere dense, then $Orb(T,x) \brc{\text{or}\,\co Orb\brc{T,x}\,\text{respectively}}$ is everywhere dense set. However, If $\D Orb(T,x)$ is somewhere dense then the situation is different.
\begin{example}\label{some not every}
Let $T\in B(\co)$ defined as $T(y)=(1/2)y$. Then, there is a vector $x\in \co$ such that $\D Orb(T,x)$ is somewhere dense in $\co$ but not everywhere dense in $\co$.
\end{example}
\begin{proof}
Let $x=1$, it is clear that $\D Orb(T,1)=\set{z: z\in \co,\, |z|\le 1}$. Then we have
 $$\left(\overline{\D Orb(T,1)}\right)^\circ =\{z: z\in \co,\, |z|< 1\}\neq \phi$$
By Example \ref{d not inv}, $\overline{\D Orb(T,1)}$ is not everywhere dense in $\co$.
\end{proof}
%================================================================================================================================================
The idea of the above example is obviously due to the fact that the disk orbit of any operator on $\co$ contains at least a non-trivial closed disk which can never be nowhere dense set. In other words, the disk orbit of any operator on $\co$ is somewhere dense in $\co$. Therefore, if the somewhere density of a disk orbit implied to the everywhere density of the disk orbit, any operator on $\co$ would be diskcyclic, which is a contradiction to Example \ref{some not every}.\\

%================================================================================================================================================
The following lemma will be our main tool to prove the next theorem.
\begin{lemma}\label{inv. closed}
If $T\in \BH$ is invertible and $$\A=\{x\in \h:\norm{x-\al_i T^{n_i}y}\to 0\, \text{for some increasing sequence} \, \{n_i\}\subset \mathbb N \, \text {and}\,\al_i\in \D \}$$ is a non trivial subset of $\h$, then $\A$ is an invariant closed subset of $\h$ under both $T$ and $T^{-1}.$
\end{lemma}
\begin{proof}
Let us choose $x\in \A$, then
\begin{eqnarray*}
\norm{T(x)-\al_iT^{n_i+1}y}&=&\left\|T\brc{\al_iT^{n_i}y-x}\right\| \\
&\leq&  \left\|T\right\|\left\|\al_iT^{n_i}y-x\right\|\to 0.
\end{eqnarray*}
It follows that $Tx\in \A$. By the same way
\begin{eqnarray*}
\norm{T^{-1}(x)-\al_iT^{n_i-1}y}&=&\norm{T^{-1}\brc{x-\al_iT^{n_i}y}}\\
&\leq& \norm{T^{-1}}\norm{x-\al_iT^{n_i}y}\to 0.
\end{eqnarray*}
Thus $T^{-1}x\in \A$.

Now we shall show that $\h\backslash \A$ is open. Let us choose $v\in \h\backslash \A$, then there is an $\epsilon >0$ such that for any large positive number $N$, we have  $\norm{v-\al_i T^ny}>\epsilon$ for all $n\geq N$ and all $\al_i\in \D$. Suppose that $B(v,\epsilon)=\{x\in \h:\norm{v-x}<\epsilon\}$. Then, it is clear that $\norm{x-\al_i T^ny}>\epsilon$ for all $x\in B(v,\epsilon)$ and so $B(v,\epsilon)\subseteq \h\backslash \A$. It follows that $\A$ is a closed set and invariant under both $T$ and $T^{-1}.$
\end{proof}
%================================================================================================================================================

\begin{theorem}\label{T and inverse}
Let $T\in \BH$ be an invertible operator, and let us suppose that we have the following properties.
\begin{enumerate}
\item \label{a5} Both $T$ and $T^{-1}$ are diskcyclic operators.
\item \label{b5} There is a vector $y$ in $\h$ such that $\overline{\D Orb(T,y)}=\overline{ \D Orb(T^{-1},y)}=\h$.
\item \label{c5} There is a vector $y$ in $\h$ such that $\overline{\set{\D Orb(T,y)\cup \D Orb(T^{-1},y)}}=\h$.
\item \label{d5} Either $T$ or $T^{-1}$ is diskcyclic operator.
\end{enumerate}
Then \ref{a5} $\Leftrightarrow$ \ref{b5} $\Rightarrow$ \ref{c5}  $\Leftrightarrow$ \ref{d5}.
\end{theorem}
\begin{proof}
\ref{a5} $\Rightarrow$ \ref{b5}: Let $T$ and $T^{-1}$ be diskcyclic operators. Since the set of all diskcyclic vectors is dense $G_\delta$, then by the Baire Category Theorem, we deduce that \ref{a5} $\Rightarrow$ \ref{b5}.\\
The implications \ref{b5} $\Rightarrow$ \ref{a5},\ref{b5} $\Rightarrow$ \ref{c5} and \ref{d5} $\Rightarrow$ \ref{c5} are trivial.\\
\ref{c5} $\Rightarrow$ \ref{d5}: By hypothesis, we have $\overline{\set{\D Orb(T,y)\cup \D Orb(T^{-1},y)}}=\h$, the Baire Category Theorem indicates that either $\overline{\D Orb(T,y)}$ or $\overline{ \D Orb(T^{-1},y)}$ has nonempty interior.
If one of them, say $\overline{\D Orb(T,y)}$ is nowhere dense, then $\overline{ \D Orb(T^{-1},y)}=\h$, and therefore $T^{-1}$ is diskcyclic. Otherwise, if $\overline{\D Orb(T,y)}$ is somewhere dense, let us find a set $\A$ exactly as in Lemma \ref{inv. closed}, $p$ be a positive integer ,\, $0\neq|\al_p|\leq 1$ and $\epsilon >0$ such that $B(\al_p T^py,\epsilon)\subseteq \overline{\D Orb(T,y)}$. It follows that there exist an increasing sequence $n_i$; $n_i \geq p \geq 0$ for all $i\ge 0$ and a sequence $\al_i\in \D$ such that
$$\norm{\al_p T^py-\al_i T^{n_i}y}\to 0 \mbox{ as } i \to \iy.$$
It follows that  $\al_p T^py\in \A$. Since $\A$ is invariant under both operators $T$ and $T^{-1}$, then $\al_p y\in T^{-p}(\A)\subseteq \A$ and $T^n(\al_p y)\in T^n(\A)\subseteq \A \mbox{ for all  }n\in \N$. Therefore,
$$\h=\overline{\set{\D Orb(T,y)\cup \D Orb(T^{-1},y)}}=\A\subseteq \overline{\D Orb(T,y)}$$
Thus, $T$ is diskcyclic. 
\end{proof}
%================================================================================================================================================
\begin{corollary}\label{suf somwhere}
Let $T$ be an invertible operator and, let $y\in \h$ such that $\D Orb(T,y)$ is somewhere dense and $\overline{\D Orb(T,y)\cup \D Orb(T^{-1},y)}=\h$. Then $\D Orb(T,y)$ is everywhere dense in $\h$.
\end{corollary}
%================================================================================================================================================
The following proposition gives us some characterizations of the spectrum of diskcyclic operators.
\begin{proposition}
Let $T\in \D C(\h)$. Then we have the following properties.
\begin{enumerate}
\item \label{a3} $\sigma(T)\cup \partial(r\D)$ is connected for some $r\geq 1$.
\item \label{b3} If $\alpha\in \sigma_p(T^*)$ then $dim Ker(T^*-\alpha)^k=1$ for all $k\geq 1$.
\end{enumerate}
\end{proposition}
\begin{proof}
 First we prove (\ref{a3}). Since $T\in SC(\h)$ then, from \cite[Proposition 3.1]{limits}, we have $\sigma(T)\cup \partial(r\D)$ is connected for some $r>0$. Moreover, from Corollary \ref{2.17}, we have $\sigma(T)\not\subset (\D)^\circ$. By \cite[Lemma 1.25]{dynamic}, it is clear that $\sigma(T)\cup \partial(r\D)$ is connected for some $r\geq 1$.
For part (\ref{b3}), the proof follows immediately from \cite[Proposition 3.1.]{limits}
\end{proof}
%================================================================================================================================================


\begin{thebibliography}{30}
\small
\bibitem{dynamic} F. Bayart, \'{E}. Matheron, \emph{Dynamics of Linear Operators}, Cambridge University Press 2009.
\bibitem{12} L. Bernal-Gonz\'{a}lez. On hypercyclic operators on Banach spaces. Proc. Amer. Math. Soc., {\bf 127} (4)(1999), 1003--1010.
\bibitem{Feldman} P. S. Bourdon and N.S. Feldman, \emph{Somewhere dense orbits are everywhere dense}, Indiana Univ. Math. J. {\bf 52 }(2003), 811--819.
\bibitem{H C Bi.} N. S. Feldman, \emph {Hypercyclicity and supercyclicity for invertible bilateral weighted shifts}, Proc. Amer. Math. Soc., {\bf 131} (2003), 479--485.
\bibitem{Erdman}  K. G. Grosse-Erdmann, A. Peris, \emph{Linear Chaos}, in: Universitext, Springer, 2011.
\bibitem{limits} D. A. Herrero. \emph{Limits of hypercyclic and supercyclic operators}, J. Funct. Anal., {\bf 99} (1991), 179--190.
\bibitem{Gerd}    G. Herzog, \emph{On linear operators having supercyclic vectors}, Studia Math., {\bf 103 }(1992), 295--298.
\bibitem{Hilden}  H. M. Hilden, L. J. Wallen, \emph{Some cyclic and non-cyclic vectors of certain operators}, Indiana Univ. Math. J. {\bf 23} (1974), 557--565.
\bibitem{Kitai} C. Kitai, \emph{Invariant Closed Sets for Linear Operators}, Thesis, University of Toronto, 1982.
\bibitem{Rolewicz}   S. Rolewicz, \emph{On orbits of elements}, Studia Math., {\bf 32} (1969), 17--22.
\bibitem{Sup & weig} H. Salas. \emph{Supercyclicity and weighted shifts}. Studia Math.,{\bf 135}(1999), 55--74.
\bibitem{cyclic}  Z. J. Zeana, \emph{Cyclic Phenomena of Operators on Hilbert Space}; Thesis, University of Baghdad, 2002.
\end{thebibliography}
\end{document}